\documentclass[12pt]{article}
\usepackage{MyStyleSheet}
\usepackage{hyperref}
\usepackage{tikz}
\usepackage{amsfonts}
\usepackage{braids}
\usepackage{blindtext}
\usepackage{geometry}
\usepackage{tikz}
\usepackage{caption}
\usepackage{theoremref}
\DeclareMathOperator{\gen}{\text{child}}
\title{Braids on the Stranded Cellular Automata Model}
\author{Alexa Renner\footnote{renneram@rose-hulman.edu; Department of Mathematics, Rose-Hulman Institute of Technology, 5500 Wabash Ave., Terre Haute, IN 47803, USA.}}
\date{September 2026}

\newtheorem{counter}{counter}[section]

\theoremstyle{definition}
\newtheorem{definition}[counter]{Definition}
\newtheorem*{definition*}{Definition}

\theoremstyle{plain}
\newtheorem{theorem}[counter]{Theorem}
\newtheorem{lemma}[counter]{Lemma}
\newtheorem{corollary}[counter]{Corollary}
\newtheorem{proposition}[counter]{Proposition}

\theoremstyle{remark}
\newtheorem{remark}[counter]{Remark}

\begin{document}
\maketitle
\newcommand{\generation}{\text{generation}}
\newcommand{\Triv}{\text{Triv}}
\newcommand{\CellWidth}{\text{CellWidth}}
\newcommand{\height}{\text{height}}
\newcommand{\pre}{\text{pre}}
\newcommand{\per}{\text{per}}
\newcommand{\post}{\text{post}}
\newcommand{\start}{\text{start}}
\newcommand{\send}{\text{end}}
\newcommand{\genTurn}{\text{gen}^\dag}
\newcommand{\ind}{\text{ind}}
\newcommand{\length}{\text{length}}
\newcommand{\ts}{\text{ts}}
\newcommand{\ct}{\text{c}}
\newcommand{\width}{\text{width}}
\newcommand{\word}{\text{word}}
\newcommand{\cs}{\text{cs}}
\newcommand{\cross}{\text{cross}}
\newcommand{\In}{\text{In}}
\begin{abstract}
    The Stranded Cellular Automata (SCA) model is a grid of cells such that each cell can contain 0, 1, or 2 strands, together with two cellular automata that control when and how strands turn and cross. It was developed to study patterns occurring in fiber arts. We define a notion of what it means for a braid, in the sense of an element of a braid group, to be represented by an SCA pattern, and provide several algorithms to determine when a braid has an SCA representation with certain additional properties.
\end{abstract}

\hfill\break
\section*{Introduction}
The Stranded Cellular Automata (SCA) model, developed in \cite{hh}, was originally designed to model fiber arts. It consists of a grid of cells, and each cell consists of some number of strands ranging from zero to two. Two cellular automata are used to determine how the strands in adjacent cells will interact with each other in the cell directly above those cells: A turning rule, which is used to determine whether, given certain configurations in a given generation, strands will turn; and a crossing rule, which is used to determine which strand, given certain configurations in a generation, will cross over the other. In \cite{gliders}, several of these intuitive notions were made rigorous.

One of the types of fiber art the SCA model has been used extensively to study is art made by crossing and knotting different strings. An example of this is \cite{loyd}, where friendship bracelet patterns were studied, and \cite{yang}, where weaving patterns were studied. The fact that many of these pieces of art use some notion of ``braiding" multiple strings together motivates the idea of representing braids on the SCA model. We develop a notion of what it means for a braid, in terms of an element of the braid group $B_n$, to have an SCA representation, and design algorithms to determine when braids have certain types of SCA representations.

Section \ref{braidgroups} provides a brief introduction to the braid groups, and section \ref{scaReps} provides some of the necessary definitions and notions surrounding braid representations and SCA patterns. In section \ref{fingridrep}, we present algorithms to determine if a pattern on the grid is an SCA representation of a specific braid, and when braids have certain types of SCA representatives, such as SCA representatives with finite height. Section \ref{compactReps} defines compact SCA representations and proves results about an algorithm for finding compact SCA representations of a braid. Finally, section \ref{outlook} provides a few avenues for future work.
\subsection*{Notation}
Throughout this report, we use the conventions and notation found in \cite{gliders}.
\begin{itemize}
    \item We use the convention $[k] = \{0,1,\cdots, k\}$.
    \item $G_n$ denotes the set of grid patterns on $n$ strands, $G_n^*$ denotes the set of grid patterns on $n$ strands that have only finitely many generations (these are called finite height grid patterns).
    \item $S_n$ denotes the set of SCA patterns on $n$ strands, $S_n^*$ denotes the set of SCA patterns on $n$ strands that have only finitely many generations.
    \item $\height:G_n^*\to\mathbb{N}$ is a function such that $\height(g)$ is the number of generations in $g$.
    \item Given lists of generations $g=[\delta_1,\cdots,\delta_n]$ and $h=[\delta_1',\cdots,\delta_m']$, we denote the list $[\delta_1,\cdots,\delta_n,\delta_1',\cdots,\delta_m']$ by $g\bullet h$, or $gh$ when there is no chance of confusion.
    \item We denote the empty list by $[]$.
    \item Given two words $w_1,w_2$, we denote the concatenation of $w_1$ and $w_2$ by $w_1*w_2$. When it is clear from context, we may write $w_1w_2$.
    \item We use the symbol ($\subset$) $\subseteq$ to denote (strict) subset, sublist, or substring. The appropriate meaning will be clear from context.
    \item Given a word $w\in W_n$, $[w]$ denotes the corresponding element of $B_n$.
\end{itemize}
We also give a quick reminder of some of the functions defined in \cite{gliders} that we will make use of:
\begin{itemize}
    \item $\gen(A, B)$: The child of two adjacent cells $A$ and $B$ in the same generation.
    \item $\length(g)$: The (possibly infinite) number of generations of a grid pattern $g$.
    \item $\width(\delta_i)$: The number of cells between the cell containing first strand and the cell containing the last strand in $\delta_i$, inclusive. 
\end{itemize}
We also use the notation $\frac{X}{YZ}$ to denote bits of turning rules corresponding to specific turning configurations, as in \cite{gliders}.
\section{Braid Groups}\label{braidgroups}
In this section, we provide a basic introduction to braid groups. 
\begin{definition}
    \cite{garside}: Let $n\in\mathbb{N}^+$. The set of strings on symbols $\sigma_1,\cdots,\sigma_{n-1},\sigma_1^{-1},\cdots,\sigma_{n-1}^{-1}$ is called $W_n$. We denote the empty string by 1. Given $w, v\in W_n$, we say $w\sim v$ if one can apply some finite sequence of the following relations to $w$ to get $v$:
    \begin{itemize}
        \item $\sigma_i\sigma_i^{-1} = 1$ (cancellation)
        \item $\sigma_i\sigma_j = \sigma_j\sigma_i$ if $|i-j|>1$ (far commutativity)
        \item $\sigma_i\sigma_{i+1}\sigma_i = \sigma_{i+1}\sigma_i\sigma_{i+1}$ (triple relation).
    \end{itemize}
    The quotient $W_n/\sim$ forms a group under the operation $[b]_\sim[c]_\sim = [bc]_\sim$, and this group is called the \textbf{braid group }$B_n$. In other words, $B_n$ is the group generated by the symbols $\sigma_1,\cdots,\sigma_{n-1},\sigma_1^{-1},\cdots,\sigma_{n-1}^{-1}\in W_n$ subject to the relations:
    \begin{itemize}
        \item $\sigma_i\sigma_i^{-1} = 1$ (cancellation)
        \item $\sigma_i\sigma_j = \sigma_j\sigma_i$ if $|i-j|>1$ (far commutativity)
        \item $\sigma_i\sigma_{i+1}\sigma_i = \sigma_{i+1}\sigma_i\sigma_{i+1}$ (triple relation)
    \end{itemize}
    for $1\leq i,j\leq n-1$.
\end{definition}
We will primarily study $B_n$ for $n\geq 3$, as $B_1 = 1$ and $B_2\cong\mathbb{Z}$. Elements of $W_n$ provide a convenient link between grid patterns and braids, so we will also study them. Each braid group (with the exception of $B_1$) is an infinite torsion-free group (see \cite{torsionfree}). As the name suggests, it can be helpful to think of elements of $B_n$ as literal braids. To do this, one can think of $\sigma_i$ as being a crossing, where the $i$th strand crosses over the $i+1$st strand, and $\sigma_i^{-1}$ being the crossing that would ``undo" $\sigma_i$, which would be the crossing where the $i+1$st strand crosses over the $i$th strand. We use this to visualize the triple relation:
\begin{center}
    \begin{tikzpicture}
    \braid a_1^{-1} a_2^{-1} a_1^{-1};
\end{tikzpicture} is equivalent to \begin{tikzpicture}
    \braid a_2^{-1} a_1^{-1} a_2^{-1};
\end{tikzpicture}
\end{center}
We will draw all of our braid diagrams going upwards. This way of visualizing braids, along with our visualization of $\sigma_i$ and $\sigma_i^{-1}$, is nonstandard but is more compatible with stranded cellular automata. 

We now consider an example in $B_4$. Notice that $\sigma_1\sigma_3\sigma_1^{-1}\sigma_2\sigma_3^{-1}\sigma_2^{-1}\sigma_1\sigma_2\sigma_1\in W_n$ is a representative of a braid in $B_4$. We claim that it represents the same braid as $\sigma_3\sigma_2\sigma_3^{-1}\sigma_1\sigma_2$. To see this, we may view the braid visually as:
\begin{center}
    \begin{tikzpicture}
    \braid a_1^{-1} a_2^{-1} a_1^{-1} a_2 a_3 a_2^{-1} a_1 a_3^{-1} a_1^{-1};
\end{tikzpicture}
\end{center}
Using far commutativity, we get:
\begin{center}
    \begin{tikzpicture}
    \braid a_1^{-1} a_2^{-1} a_1^{-1} a_2 a_3 a_2^{-1} a_3^{-1} a_1 a_1^{-1};
\end{tikzpicture}
\end{center}
Using cancellation, we get:
\begin{center}
    \begin{tikzpicture}
    \braid a_1^{-1} a_2^{-1} a_1^{-1} a_2 a_3 a_2^{-1} a_3^{-1};
\end{tikzpicture}
\end{center}
From there, we may use the triple relation to get:
\begin{center}
    \begin{tikzpicture}
    \braid a_2^{-1} a_1^{-1} a_2^{-1} a_2 a_3 a_2^{-1} a_3^{-1};
\end{tikzpicture}
\end{center}
We may use cancellation to get:
\begin{center}
    \begin{tikzpicture}
    \braid a_2^{-1} a_1^{-1} a_3 a_2^{-1} a_3^{-1};
\end{tikzpicture}
\end{center}
which is the diagram for $\sigma_3\sigma_2\sigma_3^{-1}\sigma_1\sigma_2$. Thus, the two words represent the same braid.

Notice that an element of $B_n$ can be written in multiple different ways using the generators. A common question concerning finitely presented groups is whether there exists an algorithm to tell if two representations (in our case words) of group elements (in our case, braids) represent the same element. For any $n\in\mathbb{N}^+$, the answer is yes for $B_n$: This is the main content of \cite{garside}. We will refer to this algorithm throughout the rest of this document as Garside's algorithm.
\section{SCA Representations}\label{scaReps}
We prove that $G_n$ is a partial monoid under the operation of concatenation. This will allow us to give a suitable definition of an SCA representation. 
\begin{definition}
    A \textbf{partial monoid} is a nonempty set $A$ with a partial function $\circ: A\times A\to A$ that satisfies the following:
    \begin{itemize}
        \item For $a, b, c\in A$, if $(a\circ b)\circ c$ and $a\circ(b\circ c)$ are defined, $(a\circ b)\circ c = a\circ(b\circ c)$.
        \item There exists a $1\in A$ such that, for all $a\in A$, $1\circ a$ and $a\circ 1$ are defined with $1\circ a = a\circ 1 = a$.
    \end{itemize}
\end{definition}
\begin{proposition}\thlabel{GridPartialMonoid}
    $G_n$ is a partial monoid under list composition.
\end{proposition}
\begin{proof}
    Let $a^*\in G_n, b\in G_n$, and let $a = [\delta_1,\cdots,\delta_k],b = [\delta_1',\cdots]$. We extend the definition of $\bullet: G_n^*\times G_n^*\to G_n$ to a partial function $\bullet:G_n\times G_n\to G_n$ by defining $a\bullet b = [\delta_1,\cdots,\delta_k,\delta_1',\cdots]$. Define $g\bullet[] = []\bullet g = g$ for all $g\in G_n$. Notice immediately that $\bullet$ considered as a partial function is well-defined. To show that $G_n$ is a partial monoid under $\bullet$, all that remains to show is that $\bullet$ is associative (when defined).

    Let $a, b,c\in G_n$, and suppose $(a\bullet b)\bullet c, a\bullet(b\bullet c)$ are both defined. In order for both to be defined, $a, b\in G_n^*$ and $c\in G_n$, so let $a = [\delta_1^{(1)},\delta_2^{(1)},\cdots,\delta_k^{(1)}], b = [\delta_1^{(2)},\delta_2^{(2)},\cdots,\delta_l^{(2)}],$ and $c = [\delta_1^{(3)},\delta_2^{(3)},\cdots]$. We may write:
    \begin{gather*}
        (a\bullet b)\bullet c = ([\delta_1^{(1)},\cdots,\delta_k^{(1)},\delta_1^{(2)},\cdots,\delta_l^{(2)}])\bullet[\delta_1^{(3)},\cdots] = [\delta_1^{(1)},\cdots,\delta_k^{(1)},\delta_1^{(2)},\cdots,\delta_l^{(2)},\delta_1^{(3)},\cdots]\\
        =[\delta_1^{(1)},\cdots,\delta_k^{(1)}]\bullet([\delta_1^{(2)},\cdots,\delta_l^{(2)},\delta_1^{(3)},\cdots]) = [\delta_1^{(1)},\cdots,\delta_k^{(1)}]\bullet([\delta_1^{(2)},\cdots,\delta_l^{(2)}]\bullet[\delta_1^{(3)},\cdots])\\
        = a\bullet(b\bullet c).
    \end{gather*}
    
\end{proof}
Notice that the proof of \thref{GridPartialMonoid} implies that $G_n^*$ is a monoid under list composition with identity element $[]$.

\begin{definition}
    Let $G, H$ be partial monoids. A \textbf{partial monoid homomorphism} $\varphi$ is a function $\varphi: H\to G$ such that, for $g, h\in H$ such that $gh\in H$, $\varphi(gh) = \varphi(g)\varphi(h)$ and $\varphi(g)\varphi(h)\in G$.
\end{definition}

We define:
\begin{align*}
    C = \{[s_j^{(t)}, s_{j+1}^{(t+1)}], [s_j^{(i)}, N_{j+1}], [N_j, s_{j+1}^{(i)}], [r_j^{(i)}, N_{j+1}], [N_j, l_{j+1}^{(i)}], [r_j^{(t)}, l_{j+1}^{(t+1)}], \overline{[r_j^{(t)}, l_{j+1}^{(t+1)}]}, [N_j, N_{j+1}]\\
    | j\in\mathbb{Z}, 1\leq i\leq n, 1\leq t\leq n - 1\}.
\end{align*}
Define $\nu: C\to W_n$ by
\begin{equation*}
\nu(x)=\begin{cases}
          \sigma_t \quad &\text{if} \, x =  [r_j^{(t)}, l_j^{(t+1)}] \text{ for some } 1\leq t\leq n - 1\\
        \sigma_t^{-1} \quad &\text{if} \, x =  \overline{[r_j^{(t)}, l_j^{(t+1)}]} \text{ for some } 1\leq t\leq n - 1\\
          1 \quad &\text{if} \, x \notin \{[r_j^{(t)}, l_j^{(t+1)}], \overline{[r_j^{(t)}, l_j^{(t+1)}]} : j\in\mathbb{Z}, 1\leq i\leq n, 1\leq t\leq n - 1\}\\
     \end{cases}
\end{equation*}

We define two functions, $\psi:G_n^*\to W_n$ and $\varphi: G_n^*\to B_n$, as follows:

Let $\iota: W_n\to B_n$ be the natural surjection.
\begin{align*}
    g = [\delta_1,\cdots,\delta_m]\\
    \delta_i = C_{i,0}\cdots, C_{i,w_i}\text{ for }i\in[m]\\
    \psi(g) = \nu(C_{1,0})\cdots\nu(C_{1,w_1})\cdots\nu(C_{m,0})\cdots\nu(C_{m,w_m})\\
    \varphi=\iota\circ\psi.
\end{align*}
We now show that $\varphi$ is a partial monoid homomorphism. Because $\iota$ is the natural inclusion of $G_n^*$ into $B_n$, it suffices to show that $\psi(a\bullet b) = \psi(a)\psi(b)$ for $a, b\in G_n^*$. Let $a = [\delta_1, \delta_2,\cdots, \delta_k]$ and $b = [\delta_1', \delta_2',\cdots,\delta_m']$ for some $m,k\in\mathbb{N}$. Now, suppose $\delta_i = C_{i,0}\cdots C_{i,w_i}$ for $i\in[k]$ and suppose $\delta_i' = C_{i,0}'\cdots C_{i,v_i}'$ for $i\in[m]$. We may write:
\begin{align*}
    \psi(a\bullet b) = \psi([\delta_1,\cdots,\delta_k,\delta_1',\cdots,\delta_m'])\\
    =\nu(C_{1,0})\cdots\nu(C_{1,w_1})\cdots\nu(C_{k,0})\cdots\nu(C_{k,w_k})\nu(C_{1,0}')\cdots\nu(C_{1,v_1})\cdots\nu(C_{m,0})\cdots\nu(C_{m,v_m})\\
    =\psi([\delta_1,\cdots,\delta_k])\psi([\delta_1',\cdots,\delta_m'])
    = \psi(a)\psi(b).
\end{align*}
 Then $\psi$ is alphabetic and $G_n^*, B_n$ are partial monoids, so $\varphi$ is an alphabetic partial monoid morphism. We will often abuse notation by writing $\psi(\delta)$ or $\varphi(\delta)$ instead of $\psi([\delta])$ or $\varphi([\delta])$ for a generation $\delta$. In a similar way, we will sometimes abuse notation by writing $\psi(C)$ or $\varphi(C)$ instead of $\psi([C])$ or $\varphi([C])$ for a cell $C$.
 \begin{proposition}\thlabel{PartialMonoidMorphism}
    $\varphi$ as defined above is an alphabetic partial monoid morphism.
\end{proposition}
 The definitions of grid and SCA representations are as follows:
\begin{definition}
    Let $g\in G_n^*$. $g$ is a \textbf{grid representation} of $b$ if $\psi(g)\in b$.
\end{definition}
\begin{definition}
    Let $g\in G_n^*$. $g$ is an \textbf{SCA representation} of $b$ if $\varphi(g)\in b$ and $g\in S_n$.
\end{definition}
We need a lemma on SCA patterns:
\begin{lemma}\thlabel{trcrtransfer}
    Fix $j,k\in\mathbb{N}_{\geq 2}$. If $ab^jc$ is an SCA pattern, then $ab^kc$ is an SCA pattern. Additionally, $t$ is a turning rule (crossing rule) of $ab^jc$ if and only if $t$ is a turning rule (crossing rule) of $ab^kc$.
\end{lemma}
\begin{proof}
    For contradiction, suppose there exists an SCA pattern $ab^jc$ with a turning rule $l$ and a crossing rule $v$, and $k\in\mathbb{N}_{\geq2}$ such that $ab^kc$ is not an SCA pattern under $l$ and $v$. In the following, we abuse notation by including the empty cell adjacent to $C_{t,0}$ and the empty cell adjacent to the cell containing the last strand of generation $t$. There exist cells $C_{t, j}, C_{t, j+1}$ in some generation $t$ of $ab^kc,t<\length(ab^kc)$, such that $\gen(C_{t, j}, C_{t, j+1})$ produces a bit that contradicts $v$ or $l$. Notice that at least one of the following must hold:
\begin{enumerate}
    \item There exist $C_{r, i}, C_{r, i+1}$ in $ab$ such that $\gen(C_{r, i}, C_{r, i+1}) = \gen(C_{t,j}, C_{t, j+1})$ and $C_{r, i} = C_{t,j}, C_{r,i+1} = C_{t,j+1}$.
    \item There exist $C_{r, i}, C_{r, i+1}$ in $b^2$ such that $\gen(C_{r, i}, C_{r, i+1}) = \gen(C_{t,j}, C_{t, j+1})$ and $C_{r, i} = C_{t,j}, C_{r,i+1} = C_{t,j+1}$.
    \item There exist $C_{r, i}, C_{r, i+1}$ in $bc$ such that $\gen(C_{r, i}, C_{r, i+1}) = \gen(C_{t,j}, C_{t, j+1})$ and $C_{r, i} = C_{t,j}, C_{r,i+1} = C_{t,j+1}$.
\end{enumerate}
We proceed by cases:

    \hfill\break
    \textbf{Cases 1 and 3:} Notice that the bits $C_{r, i}, C_{r,i+1}$ and $\gen(C_{r, i}, C_{r,i+1})$ contribute to $l$ and $v$ are the same bits that are generated by  $C_{t, j}, C_{t,j+1}$ and $\gen(C_{t, j}, C_{t,j+1})$ in $ab^kc$, a contradiction.

    \hfill\break
    \textbf{Case 2:} Notice that the bits $C_{r, i}, C_{r,i+1}$ and $\gen(C_{r, i}, C_{r,i+1})$ contribute to $l$ and $v$ are the same bits that are generated by  $C_{t, j}, C_{t,j+1}$ and $\gen(C_{t, j}, C_{t,j+1})$ because $k\geq 2$, a contradiction.

Thus, having reached contradictions in all cases, we may conclude that the turning rules and crossing rules of $ab^jc$ are a subset of the turning rules and crossing rules of $ab^kc$. A similar argument shows that $ab^jc$ is an SCA pattern under any turning rule and crossing rule of $ab^kc$.     
\end{proof}
\begin{remark}\thlabel{rem}
    A similar proof to that of \thref{trcrtransfer} shows that all sublists of SCA patterns are SCA patterns.
\end{remark}
\begin{lemma}\thlabel{torsionFree}
    Let $a, d, c$ be $n$-stranded grid patterns such that $\varphi(d)\neq1$, and suppose that $b\in W_n$ is a representative of a braid in $B_n$. Then, $ad^kc$ is in the same equivalence class as $b$ in $B_n$ for at most one value of $k$.
\end{lemma}
\begin{proof}
    Suppose $\varphi(ad^kc)$ and $\varphi(ad^rc)$ are equivalent to $b$ for some $k,r\in\mathbb{N}^+, k\neq r$. Without loss of generality, suppose $k>r$. Then, $[\varphi(d)^k] = [\varphi(a)^{-1}b\varphi(c)^{-1}] = [\varphi(d)^r]$. Thus, $[\varphi(d)]^{k-r}=[\varphi(d)^{k-r}] = [1]$, so $[\varphi(d)] = [1]$ as $B_n$ is torsion-free.
\end{proof}
\section{Finite Grid Representations}\label{fingridrep}
Throughout this section we denote the map defined in section \ref{scaReps} sending grid patterns to words by $\psi$ and the map defined in section \ref{scaReps} sending grid patterns to braids by $\varphi$.
\subsection{Decidability of Finite Grid SCA Representations}
Recall that $G_n^*$ denotes the set of grid patterns on $n$ strands with finite height, and $S_n^*$ denotes the set of SCA patterns on $n$ strands with finite height. We claim that the set $S = \{\langle g, b\rangle: g\in S_n, \varphi(g)\in [b]\}$ is decidable in the universe $\{\langle g, b\rangle: g\in G_n^*, b\in B_n\}$. To show this, we design a few algorithms to decide this problem. 

We first introduce two algorithms to determine whether an element of $G_n^*$ is continuous, in the sense of \cite[Subsection 2.2]{gliders}; and to determine whether or not an element of $G_n^*$ that is continuous has a crossing rule, and if it does, what the generic crossing rule of $G_n^*$ is. The notation used matches that used in \cite[Turning Rule Algorithm]{gliders}.

\hfill\break
\textbf{Algorithm 1:} (continuity algorithm):

\hfill\break
Let $g$ be a finite grid pattern with height $h$. This algorithm will return a true or false value if $g$ is continuous or if $g$ is not continuous, respectively.

Define a function $\ct$ such that:
\begin{itemize}
    \item $\ct(x_i)$ is $F$ if $x_i$ is $[r_k^{(s)}, l_{k+1}^{(s+1)}]$, $[l_k^{(s)}, r_{k+1}^{(s+1)}],$ or $[s_{k}^{(s)}, s_{k+1}^{(s+1)}]$.
    \item $\ct(x_i)$ is $L$ if $x_i$ is $[n_k^{(\varnothing)}, l_{k+1}^{(s+1)}]$.
    \item $\ct(x_i)$ is $S_1$ if $x_i$ is $[s_{k}^{(s)}, n_{k+1}^{(\varnothing)}]$.
    \item $\ct(x_i)$ is $R$ if $x_i$ is $[r_k^{(s)}, n_{k+1}^{(\varnothing)}]$.
    \item $\ct(x_i)$ is $S_2$ if $x_i$ is $[n_{k}^{(\varnothing)},s_{k+1}^{(s)}]$.
    \item $\ct(x_i)$ is $N$ if $x_i = [n_{k}^{(\varnothing)},n_{k+1}^{(\varnothing)}]$.
\end{itemize}
for some $k$ and $s$. We will use $c$ to determine if there are any subpatterns of $g$ which are not continuous. To do so, given a cell $x$, $c$ returns the data of how many strands the cell has ($F$ in the case that the cell is ``full" with two strands, $L, S_1, S_2, $ or $R$ in the case where the cell contains exactly one strand, and $N$ if the cell contains no strands) and, if the cell contains only one strand, if and how that strand turns and where the strand is in the cell. In the following, each set $C_i$ corresponds to a different configuration of cells. As an abuse of notation, we will assume that each generation of $g,\delta_j,$ is written as $\delta_j = C_{j,-1}C_{j,0}\cdots C_{j,n_j}C_{j,n_j+1}$ where $C_{j,0}$ contains the first strand of $\delta_j, C_{j,n_j}$ contains the last strand of $\delta_j$, and the cells $C_{j,-1}$ and $C_{j,n_j+1}$ are empty. Then, for all pairs of adjacent cells $x_i, x_{i+1}$ in the same generation $\delta_j$ for $j\in[h-1]$, we may write:
\begin{align*}
    C_0 = \{\ct(\gen(x_i, x_{i+1})):\ct(x_i)\in\{ F,R, S_2\}, \ct(x_{i+1}) \in\{F, L, S_1\}\}\\
    C_1 =  \{\ct(\gen(x_i, x_{i+1})):\ct(x_i) \in\{F, R,S_2\}, \ct(x_{i+1}) \in\{N, R,S_2\}\}\\
    C_2 =  \{\ct(\gen(x_i, x_{i+1})):\ct(x_i) \in\{N, L,S_1\}, \ct(x_{i+1}) \in\{F, L,S_1\}\}\\
    C_3 =  \{\ct(\gen(x_i, x_{i+1})):\ct(x_i) \in\{N,S_1,L\}, \ct(x_{i+1}) \in\{N, S_2, R\}\}.
\end{align*}
From here, we need to determine if $g$ contains any discontinuities. We define the following variables $s_0,s_1,s_2,s_3$:
\begin{enumerate}
    \item If $C_0\subseteq\{F\}$, let $s_0 = 0$. Otherwise, let $s_0 = 1$.
    \item If $C_1\subseteq\{R, S_1\}$, let $s_1 = 0$. Otherwise, let $s_1 = 1$.
    \item If $C_2\subseteq\{L,S_2\},$ let $s_2 = 0$. Otherwise, let $s_2 = 1$.
    \item If $C_3\subseteq\{N\},$ let $s_3 = 0,$ otherwise let $s_3 = 1$.
\end{enumerate}
If $s_i = 0$ for $0\leq i\leq 3$, return true, otherwise return false. Notice that by the definition given in \cite{gliders}, $g$ is continuous if and only if $s_i=0$ for $0\leq i\leq 3$.

\hfill\break
\textbf{Algorithm 2:} (crossing rule algorithm):

\hfill\break
Let $g$ be a continuous finite grid pattern. This algorithm will return a generic crossing rule if $g$ has a crossing rule and the empty set otherwise. 

Define a function $\cs$ on the set of continuous grid pattern such that:
\begin{itemize}
    \item $\cs(x_i)$ is $T_1$ if $x_i = [r_k^{(s)}, l_{k+1}^{(s+1)}]$
    \item $\cs(x_i)$ is $T_2$ if $x_i = \overline{[r_k^{(s)}, l_{k+1}^{(s+1)}]}$
    \item $\cs(x_i)$ is $L$ if $x_i = [n_k^{(\varnothing)}, l_{k+1}^{(s+1)}]$ or $[s_k^{(s)},n_{k+1}^{(\varnothing)}]$.
    \item $\cs(x_i)$ is $R$ if $x_i = [r_k^{(s)}, n_{k+1}^{(\varnothing)}]$ or $[n_{k}^{(\varnothing)},s_{k+1}^{(s)}]$.
    \item $\cs(x_i)$ is $S$ if $x_i = [s_k^{(s)},s_{k+1}^{(s+1)}]$.
\end{itemize}
for some $k$ and $s$. Intuitively, $\cs$ extracts the information necessary from each pair of cells and their child to determine the bit in the crossing rule that set of cells contributes. In the following, each set $B_i$ corresponds to a bit in the crossing rule. We may write: 
\begin{align*}
    B_0 = \{\cs(\gen(x_i, x_{i+1})):\cs(x_i) = T_2, \cs(x_{i+1}) = T_1\}\\
    B_1 =  \{\cs(\gen(x_i, x_{i+1})):\cs(x_i) = T_2, \cs(x_{i+1}) \in\{ L, S\}\}\\
    B_2 =  \{\cs(\gen(x_i, x_{i+1})):\cs(x_i) = T_2, \cs(x_{i+1}) = T_2\}\\
    B_3 =  \{\cs(\gen(x_i, x_{i+1})):\cs(x_i) \in\{R,S\}, \cs(x_{i+1}) = T_1\}\\
    B_4 =  \{\cs(\gen(x_i, x_{i+1})):\cs(x_i) \in\{ R,S\}, \cs(x_{i+1})\in\{ L,S\}\}\\
    B_5 =  \{\cs(\gen(x_i, x_{i+1})):\cs(x_i) \in\{R,S\}, \cs(x_{i+1}) = T_2\}\\
    B_6 =  \{\cs(\gen(x_i, x_{i+1})):\cs(x_i) = T_1, \cs(x_{i+1}) = T_1\}\\
    B_7 =  \{\cs(\gen(x_i, x_{i+1})):\cs(x_i) = T_1, \cs(x_{i+1}) \in\{ L,S\}\}\\
    B_8 =  \{\cs(\gen(x_i, x_{i+1})):\cs(x_i) = T_1, \cs(x_{i+1}) = T_2\}.
\end{align*}
From here, we must determine all possible crossing rules of $g$, if any exist. Note that, by the assumed continuity of $g$, $B_i\subseteq\{T_1,T_2\}$ for $0\leq i\leq 8$. If $B_i$ contains $T_1$ and $T_2$ for any $i\in\{0,1,\cdots,8\},$ then no crossing rule can exist. If this is the case, return $\varnothing$. However, notice that if $g$ has a crossing rule and $B_i$ does not contain $T_1$ or $T_2$, then there exist crossing rules of $g$ with a 1 in position $i$ and other crossing rules of $g$ with a 0 in position $i$. Thus, in the algorithm, first determine if there exists an $i\in \{0,1,\cdots,8\}$ with $T_1,T_2\in B_i$. If so, then $g$ does not have a crossing rule. If not, for each $i\in \{0,1,\cdots,8\}$, define:
\begin{itemize}
    \item $s_i = 0$ if $T_2\in B_i, T_2\not\in B_i$.
    \item $s_i = 1$ if $T_1\in B_i, T_2\not\in B_i$
    \item $s_i = X$ if $T_1,T_2\not\in B_i$.
\end{itemize}
Let $s = s_0s_1s_2s_3s_4s_5s_6s_7s_8$. Return $s$. By definition, the generic crossing rule of $g$ is $s$. 
\begin{proposition}\thlabel{sn}
    $S_n^*$ is decidable in $G_n^*$.
\end{proposition}
\begin{proof}
    Let $g\in G_n^*$. Notice that $g\in S_n^*$ if and only if algorithms 1, 2, and \cite[Turning Rule Algorithm]{gliders} accept $g$, so run algorithm 1 on $g$. If it rejects $g$, reject; and if it accepts $g$, run algorithms 2 and \cite[Turning Rule Algorithm]{gliders} on $g$. If either one rejects, reject; otherwise accept.
\end{proof}
Let $M$ be the algorithm to decide whether a grid pattern is an SCA pattern given in \thref{sn}. We now describe an algorithm to determine if, given a grid pattern and a braid, the grid pattern is an SCA representation of the braid.

\hfill\break
\textbf{Algorithm 3} (finite grid pattern algorithm): 

\hfill\break
Given a grid pattern $g\in G_n^*$ and a braid representative $w\in W_n$ of a braid $B_n$:
\begin{itemize}
    \item[1.] Run $M$ on $g$. If $M$ rejects, reject.
    \item[2.] Compute $\psi(g)$.
    \item[3.] Run Garside's algorithm on $\psi(g)$ and $w$ to determine if $\psi(g)$ and $w$ represent the same braid. If it accepts, accept; otherwise reject.
\end{itemize}
\begin{theorem}
    The problem of whether $g\in G_n^*$ is an SCA representation for a braid in $B_n$ represented by word $w$ is decidable.
\end{theorem}
\begin{proof}
    By \thref{sn}, there exists a decider for the problem of whether $g\in S_n^*$. Run that decider on $g$. If it rejects, reject. If it accepts, apply algorithm 3 to $\langle g, w\rangle$. If it accepts, accept; otherwise, reject.
\end{proof}
\subsection{Decidability of the Existence of Fixed Height SCA Representations}\label{height}
In this section, we design an algorithm to determine if, given a $h\in\mathbb{N}^+$ and a braid representative $w$, there exists an SCA pattern of height $h$ that is an SCA representation of $[w]$. We first need some new definitions:

Given a grid pattern $g$, we may order the cells of $g$ using the lexicographical ordering, that is, $C_{i,j}<C_{a, b}$ if $i<a$ or if $i=a$ and $j<b$. Notice this is a well ordering. Given a grid pattern $g$, we may assign a new (not necessarily continuous) grid pattern $g^<$ to $g$ defined as $[x_1,x_2,\cdots]$ such that $x_i<x_{i+1}$ for all $i$, each $x_i$ is a cell of $g$, and all cells of $g$ are included in $g^<$. Given $g^{<} = [x_1,\cdots, x_u]$, we define the initial segment of $g^{<}$ prior to $x_t$ by $g^{<x_t}$.

As an example, consider the grid pattern $g = [\delta_1,\delta_2,\delta_3]$ with $\delta_1 = \overline{[r_1^{(1)},l_1^{(1)}]},[r_3^{(3)},l_4^{(4)}][s_5^{(5)},n_6^{(\varnothing)}]$, $\delta_2 = [n_0^{(\varnothing)},s_1^{(1)}][r_2^{(2)},l_3^{(3)}][s_4^{(4)},s_5^{(5)}]$, $\delta_3 = [r_1^{(1)},l_2^{(2)}]\overline{[r_3^{(3)},l_4^{(4)}]}[s_5^{(5)},n_6^{(\varnothing)}]$ below:
\begin{center}
    \includegraphics[scale=0.7]{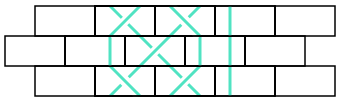}
\end{center}
Then, 
\begin{gather}
    g^{<C_{2,3}} = [\overline{[r_1^{(1)},l_1^{(1)}]},[r_3^{(3)},l_4^{(4)}],[s_5^{(5)},n_6^{(\varnothing)}],[n_0^{(\varnothing)},s_1^{(1)}],[r_2^{(2)},l_3^{(3)}]].
\end{gather}
\begin{definition}
    Let $g$ be a grid pattern, and let $C_{i,j}$ be a cell in some generation $\delta_i$ with $i>1$. Then, the \textbf{parents} of $C_{i,j}$ is the pair $(C_{i-1,k},C_{i-1,k+1})$ for the $k$ such that $\gen(C_{i-1,k},C_{i-1,k+1}) = C_{i,j}$.
\end{definition}
We now introduce some new notation. Let $g\in G_n$ be a grid pattern. We denote the (possibly empty) set of turning rules of $g$ by $T_g$ and the (possibly empty) set of crossing rules of $g$ by $C_g$.
\begin{definition}
    Let $N\in\mathbb{N}^+$. Let $g_1,\cdots, g_k$ be finite SCA patterns of the same height such that, for each $i\in[k], g_i\in S_{n_i}^*$ and $g_i = [\delta_{1,i},\cdots,\delta_{N,i}]$. Further suppose that $T=\bigcap_{i\in[k]} T_{g_i}\neq\varnothing$ and $C=\bigcap_{i\in[k]} C_{g_i}\neq\varnothing$. Let $u$ be the width $N$ generation of empty cells. Let $t\in T,c\in C$. Then, the \textbf{product} $\prod_{i\in[k]}^{c,t}g_i$ with respect to $c$ and $t$ is the length $N$ SCA pattern generated by $\delta_1 = [\delta_{1,1},u,\delta_{1,2},u,\cdots,\delta_{1,k}]$ under $t$ and $c$.
\end{definition}
\begin{proposition}\thlabel{independentofChoice}
    Let $c_1,c_2\in C, t_1,t_2\in T$. Then, $\prod_{i\in[k]}^{c_1,t_1}g_i = \prod_{i\in[k]}^{c_2,t_2}g_i$.
\end{proposition}
\begin{proof}
    Let $g$ be the product with respect to $c,t$ for any $c\in C, t\in T$. Consider generation $\delta_j$ of $g,j\leq N$. Let $C_{j,i},C_{j,t}$ be two cells belonging to two patterns $g_r, g_m$ in the product, $m\neq r$ Then, by type A and type B generations and the choice of $j,$ there are at least $N - j + 1$ cells between $C_{j,i}$ and $C_{j,t}$.

    Let $c_1,c_2\in C, t_1,t_2\in T$. Let $g=\prod_{i\in[k]}^{c_1,t_1}g_i$ and let $g'=\prod_{i\in[k]}^{c_2,t_2}g_i$. Refer to the $j$th cell in the $i$th generation of $g$ as $C_{i,j}$ and the $j$th cell in the $i$th generation of $g'$ as $C_{i,j}'$. For contradiction, suppose there exists some cell $C_{i,j}\neq C_{i,j}'$. Notice that $i \neq 1$ as $C_{1,j} = C_{1,j}'$ for all $j$ by construction, so we may assume $i>1$. Choose the minimal (under the ordering described in \thref{orderingdef}) such cell, $C_{i,j}$. Then, because $i>1$, there exist cells $C_{i-1,t},C_{i-1,t+1}$ in generation $i-1$ of $g$ such that $\gen(C_{i-1,t},C_{i-1,t+1}) = C_{i,j}$. 

    By minimality, $C_{i-1,t} = C_{i-1,t}'$ and $C_{i-1,t+1} = C_{i-1,t+1}'$. Notice that $C_{i-1,t}$ or $C_{i-1,t+1}$ must be nonempty, or else $C_{i,j},C_{i,j}'$ would both be empty. If $C_{i-1,t}$ is nonempty, it must belong to $g_r$ for some $r\in[k]$. Then, by the fact stated at the beginning of the proof, $C_{i-1,t+1}$ must belong to $g_r$ or must be empty (if $C_{i-1,t}$ in $g_r$ is the last cell containing a strand). A similar argument shows that if $C_{i-1,t+1}$ is nonempty and belongs to $g_r, C_{i-1,t}$ belongs to $g_r$ or is empty (if $C_{i-1,t}$ in $g_r$ is the first cell containing a strand). Thus, under any crossing and turning rules of $g_r, \gen(C_{i-1,t},C_{i-1,t+1}) = C_{i,j}$. Because $t_1,t_2\in T\subseteq T_{g_r}, c_1,c_2\in C\subseteq C_{g_r},$ we have $C_{i,j} = \gen(C_{i-1,t},C_{i-1,t+1}) =\gen(C_{i-1,t}',C_{i-1,t+1}') = C_{i,j}'$, a contradiction.
\end{proof}
In light of \thref{independentofChoice}, given $g_1,\cdots, g_k$ finite SCA patterns of the same height such that $\bigcap_{i\in[k]} T_{g_i}\neq\varnothing$ and $\bigcap_{i\in[k]} C_{g_i}\neq\varnothing$, we may refer to the product $\prod_{i\in[k]}g_i$ of $g_1,\cdots,g_k$ without reference to a specific turning or crossing rule.

\begin{proposition}
    Let $g_1,\cdots, g_k$ be finite SCA patterns of the same height such that $\bigcap_{i\in[k]} T_{g_i}\neq\varnothing$ and $\bigcap_{i\in[k]} C_{g_i}\neq\varnothing$. Then, for all $i\in[k], g_i$ is a subpattern of $\prod_{i\in[k]}g_i$.
\end{proposition}
\begin{proof}
   Let $i\in[k]$. Suppose $g_i$ contains strands $s,s+1,\cdots, t$ in $\prod_{i\in[k]}g_i$. Let $g_i = [\delta_1',\cdots,\delta_N']$. For each $\delta_j$ of $g=\prod_{i\in[k]}g_i$, call the subpattern of $\delta_j$ containing strands $s,s+1,\cdots,t$ $p_i$. We claim that $p_j = \delta_j'$ for all $j\in[N]$. 

    We proceed by induction on $j$. The base case is satisfied by definition.

    \hfill\break
    \textbf{Inductive Step:} Suppose $p_j = \delta_j'$ for some $j<N$. Then, notice that strand $s-1$ (if it exists) is separated from strand $s$ by at least $N - j$ empty cells in $\delta_{j+1}$ and strand $t$ is separated from strand $t+1$ (if it exists) by at least $N-j$ empty cells in $\delta_{j+1}$. Now, let $C_{j+1,r}$ be a cell in $p_{j+1}$. Notice that there exist cells $C_{j,w}, C_{j,w+1}$ in $\delta_j$ such that $\gen(C_{j,w},C_{j,w+1}) = C_{j+1,r}$ and such that one of the following is satisfied:
    \begin{enumerate}
        \item $C_{j,w+1}$ and $C_{j,w}$ are nonempty and each contain a strand in $p_j$.
        \item $C_{j,w+1}$ is empty but $C_{j,w}$ contains a strand that is not $t$ or $s$.
        \item $C_{j,w}$ is empty but $C_{j,w+1}$ contains a strand that is not $t$ or $s$.
        \item $C_{j,w+1},C_{j,w}$ are empty.
        \item $C_{j,w}$ contains strand $t$.
        \item $C_{j,w+1}$ contains strand $s$.
    \end{enumerate}
    Notice that by the inductive hypothesis, $p_j = \delta_j'$. Also notice that strand $s-1$ (if it exists) in $\delta_j$ is separated from strand $s$ by at least $N-j+1$ empty cells and strand $t+1$ (if it exists) in $\delta_j$ is separated from strand $t$ by at least $N-j+1$ empty cells.
    
    Then, in cases (1)-(3), $C_{j,w},C_{j,w+1}$ are both in $\delta_j'$. In case (4), $\gen(C_{j,w+1},C_{j,w})$ must be empty by continuity of $g$ and $g'$. In case (5), $C_{j,w}$ is in $\delta_j'$ and $C_{j,w+1}$ is empty, and in case (6), $C_{j,w+1}$ is in $\delta_j'$ and $C_{j,w}$ is empty. Call the cell in $g_r$ corresponding to $C_{j,w}$ $A$ and the cell in $g_r$ corresponding to $C_{j,w+1}$ $B$. Notice that in all cases, the cell $\gen(A,B)$ is identical to the cell $\gen(C_{j,w},C_{j,w+1})$, as $T\subseteq T_{g_r}$ and $C\subseteq C_{g_r}$. Thus, $p_{j+1} = \delta_{j+1}'$, and the inductive hypothesis is confirmed.
\end{proof}
\begin{definition}
    Let $g$ be a finite continuous $n$-stranded grid pattern for some $n\in\mathbb{N}^+$. Then, there exist $h_1,\cdots,h_r$ such that for each $i\in[r]$, and each generation $\delta_{j}$ of $h_{i}, \width(
    \delta_j) = \width(h_i)$, such that $\prod(h_1,\cdots,h_r) = g$. Let $(h_1,\cdots,h_r)$ be a maximal such tuple, and let $(g_1,\cdots,g_k)$ be the subtuple containing only the elements of $(h_1,\cdots,h_r)$ which contain a strand. The tuple $(g_1,\cdots,g_k)$ is called the \textbf{block decomposition} of $g$, and each $g_i$ for $i\in[k]$ is called a \textbf{block}. If the block decomposition of $g$ is $(g)$, we call $g$ a \textbf{block}. 
\end{definition}
\begin{lemma}\thlabel{scablockdecomp}
    Suppose $g$ has block decomposition $(g_1,\cdots,g_r)$. Then, for any $h_0,\cdots,h_{r}$ such that for all $1\leq i\leq r$, $h_i$ is a pattern of empty cells of constant width or $[]$; $\psi\left(\prod (g_1,\cdots, g_r)\right) = \psi\left(\prod (h_0,g_1,h_1,\cdots, g_r,h_r)\right)$ and $g$ is an SCA pattern if and only if $\prod (h_0,g_1,h_1,\cdots, g_r,h_r)$ is.
\end{lemma}
\begin{proof}
First, notice that any empty grid pattern has all possible turning and crossing rules. For each $0\leq i\leq r$, denote the $j$th generation of $h_i$ by $\delta_{i, j}'$. To see the first claim, notice that:
\begin{gather*}
    \psi\left(\prod (g_1,\cdots, g_r)\right) = \psi(\delta_{1,1})\psi(h)\cdots\psi(h)\psi(\delta_{r,\height(g_r)})\\
=\psi(\delta_{0,1}')\psi(h)\psi(\delta_{1,1})\psi(h)\cdots\psi(h)\psi(\delta_{r,\height(g)})\psi(h)\psi(\delta_{r,\height(g)'}) =\psi\left(\prod (h_0,g_1,h_1,\cdots, g_r,h_r)\right).
\end{gather*}

Now, suppose $g$ is an SCA pattern. First, note that $\prod (h_0,g_1,h_1,\cdots, g_r,h_r)$ is continuous and has the same height as $g$. For contradiction, suppose $\prod (h_0,g_1,h_1,\cdots, g_r,h_r)$ is not an SCA pattern. Then, there exist cells $C_{i, j}, C_{i, j+1}, i<\height(g)$ in some generation of $\prod (h_0,g_1,h_1,\cdots, g_r,h_r)$ such that the bit generated in $\gen(C_{i, j}, C_{i, j+1})$ conflicts with all crossing and turning rules of $g$. Note that $C_{i, j}, C_{i,j+1}$ cannot be empty cells, or else $\prod (h_0,g_1,h_1,\cdots, g_r,h_r)$ would not be continuous. Also notice that $C_{i,j}, C_{i,j+1}$ cannot be neighboring cells in some $g_u,u\in[r]$, or else the bit they generate would appear in all crossing and turning rules of $g$. Thus, we have two cases:

\hfill\break
\textbf{Case 1:} $C_{i,j}$ is in $g_u$ for some $u\in[r]$ and $C_{i,j+1}$ is empty. Notice then that these two cells also appear next to each other in generation $i$ of $g$, so the bit they generate appears in all crossing and turning rules of $g$, a contradiction.

\hfill\break
\textbf{Case 2:} $C_{i,j+1}$ is empty and $C_{i,j+1}$ is in $g_u$ for some $u\in[r]$. This case follows almost identically to case 1.

Thus, if $g$ is an SCA pattern, then $\prod (h_0,g_1,h_1,\cdots, g_r,h_r)$ is an SCA pattern. The proof of the converse is almost identical.
\end{proof}




\begin{theorem}
    Let $h, n\in\mathbb{N}^+$, and let $w\in W_n$ be a representative of $b\in B_n$. The problem of whether there exists a height $h$ SCA representation $g\in S_n$ of $b$ is decidable.
\end{theorem}
\begin{proof}
    Let $M$ be the machine that decides whether or not a pattern is continuous (\thref{sn}). Given a representative $w\in W_n$ of a braid in $b\in B_n$ and a $h\in\mathbb{N}$, we may run the following algorithm:
\begin{enumerate}
    \item If $h = 0$, accept.
    \item For $1\leq k\leq n$, enumerate all continuous patterns of at most height $h$ on $k$ strands that have one block.
    \item Enumerate all possible ordered tuples $(g_{1,1},\cdots,g_{1,r}),\cdots,(g_{t,1},\cdots,g_{t,r})$ of patterns listed in (1) such that the sum of the number of strands in each pattern is $n$.
    \item For each $i\in[t]$:
    \begin{enumerate}
        \item Determine $\height(g_{i,j})$ for each $j\in[r]$. If any two heights differ, continue with $i + 1$ if $i<t$ (if $i=t$, reject).
        \item For each $j\in[r]$:
        \begin{enumerate}
            \item Run $M$ on $g_{i,j}$. If it accepts, form the set $T_{g_{i,j}}$ of turning rules of $g_{i,j}$ and the set $C_{g_{i,j}}$ of crossing rules of $g_{i,j}$. If it rejects, continue with $i + 1$ if $i<t$ (if $i=t$, reject).
        \end{enumerate}
        \item Determine whether one of $\bigcap_{j\in[r]}T_{g_{i,j}}$ or $\bigcap_{j\in[r]}T_{g_{i,j}}$ is empty. If one is, continue with $i + 1$ if $i<t$ (if $i=t$, reject). 
        \item Compute $\prod_{j\in [r]}g_{i,j}$.
        \item Compute $\psi\left(\prod_{j\in [r]}g_{i,j}\right)$.
        \item Run Garside's algorithm on $\psi\left(\prod_{j\in [r]}g_{i,j}\right)$ and $w$. If it accepts, accept; otherwise if $i<t$, continue with $i+1$. If $i = t$, reject.
    \end{enumerate}
\end{enumerate}
We now prove that the algorithm works as desired. Notice that in step (2), there are only finitely many such patterns by the definition of a block, so the algorithm halts on all inputs. Suppose $g$ is an SCA representative for $[w]=b$. We may write $g$ in block form as $(g_1,\cdots, g_r)$ for some $r\leq n$. Then, $(g_1,\cdots, g_r)$ is enumerated in the sequence in step (2). Notice that there exist $h_0,\cdots, h_r$ such that $g = \prod (h_0,g_1,h_1,\cdots,g_r,h_r)$, but by \thref{scablockdecomp}, this implies that $\prod(g_1,\cdots,g_r)$ is an SCA pattern. Because $\psi\left(\prod(g_1,\cdots,g_r)\right) =\psi\left(\prod (h_0,g_1,h_1,\cdots,g_r,h_r)\right) = \psi(g)$ and because $g$ is an SCA representation of $b$, the algorithm accepts in stage 3(d).

If the algorithm accepts, it has constructed an SCA representative for $[w]=b$.
\end{proof}

\subsection{Decidability of Existence of Fixed Width SCA Representations}\label{babyOrbitProblem}
In this section, we design an algorithm to determine whether, given a braid representative $b\in W_n$ and a $w\in\mathbb{N}$, there exists an SCA representative of $[b]$. To do this, we first define some machines:

Let $M$ be a machine that determines if a pattern is an SCA pattern, which exists by \thref{sn}. By \cite[Corollary 2.5]{orbitproblem}, the problem of whether given $w,v\in W_n$ there exists a $k\in\mathbb{N}$ such that $[w]^k = [v]$ is decidable. Let $A$ be a decider for this problem.
\begin{theorem}\thlabel{existancefixedwidthdecidability}
   Fix $n\in\mathbb{N}^+$. The problem of whether, given $\langle b, w\rangle$ with $b\in W_n$ and $w\in\mathbb{N}$, there exists an SCA representation of $[b]$ is decidable.
\end{theorem}
\begin{proof}
We present an algorithm that we claim decides this problem:

\hfill\break
\textbf{Algorithm For Existence} (of a fixed width representation):

\hfill\break
On input $\langle b,w\rangle$, for $b\in W_n$ and $w\in\mathbb{N}$:
\hfill\break
\begin{enumerate}
    \item Enumerate all possible grid patterns on $n$ strands of width $w$ and height at most $8^w$.
    \hfill\break
    ///Notice that we choose $8^w$ in the above because any SCA pattern of width at most $w$ must repeat after $8^w$ generations.
    \item Run $M$ on each pattern and delete the patterns that are not SCA patterns from the enumeration. If the enumeration is empty, reject.
    \item Of the remaining patterns, enumerate all ordered 3-tuples of patterns from the enumeration as $(a_1, d_1, c_1),\cdots, (a_l,d_l,c_l)$.
    \item For $1\leq i< l$:
    \begin{enumerate}
        \item Compute $\psi(a_i), \psi(d_i),\psi(c_i)$ and form $(\psi(d_i), \psi(a_i)^{-1}b\psi(c_i)^{-1})$.
        \item Run $A$ on $\langle \psi(d_i), \psi(a_i)^{-1}b\psi(c_i)^{-1}\rangle$. If it rejects, move on to $i +1$. If it accepts:
        \begin{enumerate}
            \item Run $M$ on $a_ic_i$, $a_id_ic_i$, and $a_id_i^2c_i$. If all three patterns receive a rejection, move on to $i + 1$. If no pattern is rejected by $M$ machine, accept. 
            \item If $a_id_i^2c_i$ is rejected by $M$ but the elements of some nonempty subset $P\subseteq\{a_ic_i,a_id_ic_i\}$ are the elements of $\{a_ic_i,a_id_ic_i\}$ accepted by $M$:
            \begin{enumerate}
                \item For each $p\in P$, run Garside's algorithm on $p$ and $b$. If it accepts for some $p\in P$, accept; otherwise move to stage $i+1$.
            \end{enumerate}
            \item If $a_ic_i$ is rejected by $M$ but $a_id_ic_i$ and $a_id_i^2c_i$ are accepted by $M$:
            \begin{enumerate}
                \item Run Garside's algorithm on $\psi(a_ic_i)$ and $b$. If it rejects, accept. If it accepts:
                \item Run Garside's algorithm on $\psi(d_i)$ and $1$. If it accepts; accept. If it rejects, move on to stage $i + 1$.
            \end{enumerate}
            \item If $a_id_ic_i$ is rejected but $a_ic_i$ and $a_id_i^2c_i$ are accepted:
            \begin{enumerate}
                \item Run Garside's algorithm on $\psi(d_i)$ and $1$. If it accepts, accept. Otherwise:
                \item Run Garside's algorithm on $\psi(a_ic_i)$ and $b$. If it accepts, accept, otherwise:
               \item Run Garside's algorithm on $\psi(d_i)$ and $\psi(a_i)^{-1}b\psi(c_i)^{-1}$. If it accepts, move to stage $i+1$; if it rejects, accept.
            \end{enumerate}
            \item If $a_ic_i$ and $a_id_ic_i$ are rejected and $a_id_i^2c_i$ is accepted:
            \begin{enumerate}
                \item Run Garside's algorithm on $\psi(d_i)$ and $1$. If it accepts, accept; otherwise:
                    \item Run Garside's algorithm on $\psi(d_i)$ and $1$. If it accepts, accept, otherwise:
                        \item Run Garside's algorithm on $1$ and $\psi(a_i)^{-1}b\psi(c_i)^{-1}$. If it accepts, accept. Otherwise:
                     
                            \item Run Garside's algorithm on $\psi(d_i)$ and $\psi(a_i)^{-1}b\psi(c_i)^{-1}$. If it accepts, accept. Otherwise, move on to stage $i+1$.
            \end{enumerate}
            \end{enumerate}
    \end{enumerate}
    \item If the machine reaches this stage, reject.
\end{enumerate}
Notice that because $\psi$ is computable and because Garside's algorithm always halts, the above algorithm always halts. We claim that the algorithm accepts only the appropriate input. Suppose $[b]$ has a representative of width at most $w$, call it $g$. Suppose $g = [\delta_1,\cdots,\delta_m]$. If $\length(g)\leq8^w$, let $a = g$ and if $\length(g)>8^w$, let $a$ be the first $8^w$ generations of $g$. Notice that if there exist more than $8^w$ remaining generations, there must exist a sublist $d$ of $g$ such that $ad^k$ is a sublist of $g$ for some $k\geq 1$ such that $\length(g) - \length(ad^k)<\length(d)$ by the fact that there are $8^w$ possible generations of width at most $w$. In the case where such generations do not exist, let $d =[]$. Let $c$ be $[\delta_{\length(ad^k)+1},\cdots,\delta_m]$. Notice that $M$ accepts $a, d, c$ (see \thref{rem}). Then, because there exists some $k\in\mathbb{N}^+$ such that $[ad^kc] = [b]$, $A$ accepts $\langle \psi(d),\psi(a)^{-1}b\psi(c)\rangle$. If $k =1$, then the algorithm accepts in one of (b)(i), (b)(ii), or (b)(iii). If $k >1$, then the algorithm accepts in one of (b)(i), b(iii), (b)(iv), or (b)(v) by \thref{trcrtransfer} and \thref{torsionFree}. 

Now suppose the machine accepts. Then, there exists a tuple $(a, d, c)$ that is accepted in stage 4. If it is accepted in stage (b)(i), then $ad^kc$ is an SCA pattern for all $k\in\mathbb{N}$ by \thref{trcrtransfer} and there exists a $k\in\mathbb{N}$ such that $\psi(ad^kc)=\psi(a)\psi(d)^k\psi(c)$ is equivalent to $b$ under the braid group relations, so $ad^kc$ is an SCA representation of $[b]$. If the tuple is accepted in (b)(ii), then one of $ac$ or $adc$ is an SCA representation of $[b]$. If the tuple is accepted in (b)(iii), then $ad^kc$ is an SCA representation of $[b]$ for some $k\in\mathbb{N}^+$ by \thref{trcrtransfer}. If the tuple is accepted in (b)(iv), then $ac$ or $ad^kc$ is an SCA representation of $[b]$ for some $k\geq 2$ by \thref{trcrtransfer}. Finally, if the tuple is accepted in (b)(v), $ab^kc$ is an SCA representation of $[b]$ for some $k\in\mathbb{N}_{\geq 2}$ by \thref{trcrtransfer}. Thus, in all cases, $[b]$ has an SCA representation with width at most $w$.
\end{proof}
\section{Compact Representations}\label{compactReps}
We first introduce a few definitions.
\begin{definition}
    A $g\in G_n^*$ is \textbf{compact} if, for each generation $\delta_i$ of $g$, the position of the leftmost strand subtracted from the position of the rightmost strand is $n - 1$.
\end{definition}
Intuitively, a compact grid pattern is a pattern where the strands are always as ``close" together as possible.
\begin{definition}
    A word $w\in W_n$ is \textbf{compactly SCA-representable} if there exists a compact $g\in S_n^*$ such that $\psi(g) = w$.
\end{definition}
Intuitively, a compactly SCA-representable word is a word that is representable in the ``tamest", or most intuitive, possible way, in the sense that the strands in the representation need not have gaps between them.
\begin{definition}
    A braid $b\in B_n$ is \textbf{compactly SCA-representable} if $b$ has an SCA representation $s\in S_n^*$ that is compact. We call the maximal set of elements $S'$ of $S_n^*$ such that each $s\in S'$ is compact and $\psi[S']\subseteq b$ the compact SCA representation set of $b$.
\end{definition}
 Notice immediately that the following can be proved by modifying the proof of \thref{existancefixedwidthdecidability} slightly:
 \begin{corollary}
     Let $b = [w]\in B_n$. Then, the problem of whether $b$ has a compact SCA representation is decidable.
 \end{corollary}

We present results that characterize the structure of compact SCA representations of a given word, along with results that can be used to prove that a set of braids is compactly SCA representable by only proving that a single word is compactly SCA representable.


To start, we need a bit more terminology:
\begin{definition}
  A generation $k$ is \textbf{type A} if the cell $C_{k, 0}$ contains exactly two strands and \textbf{type B} if this is not the case.  
\end{definition}
Recall that $C_{k,0}$ must always contain at least one strand, so a generation $k$ is type B if $C_{k,0}$ contains exactly one strand.

Notice that in a compact grid pattern, $v_j^{(i)}$ must have $j = i$ for $v\in\{s, r, l\}$. Consider a compact grid pattern $g$. If there exists a crossing in some generation $\delta_i$ of $g$ on strands $j, j+1$ such that $j$ is odd, $\delta_i$ must be a type A generation. To see this, notice that if $C_{j,0}$ contains two strands, strand $j-1$ would be contained in the same cell as strand $j$, a contradiction, as strand $j+1$ would be contained in a different cell. Similarly, if there is a crossing on strands $j, j+1$ such that $j$ is even in some generation $\delta_i$, $\delta_i$ must be a type B generation. The following graphic demonstrates the difference between type A and type B generations:

\begin{center}
    \includegraphics[scale=0.7]{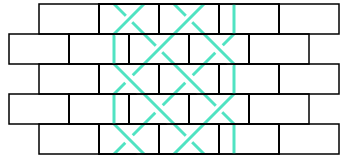}
\end{center}
We can see from the figure that $[r_2^{(2)}, l_3^{(3)}]$ or $[r_4^{(4)}, l_5^{(5)}]$ could not occur in generations 1, 3, or 5 -- the type A generations of this pattern. We can also see that $[r_1^{(1)}, l_2^{(2)}]$ or $[r_3^{(3)}, l_4^{(4)}]$ could not occur in generations 2 or 4 -- the type B generations of this pattern.

\subsection{$\gamma$ Algorithm}
In this section, we describe an algorithm for converting words in $W_n$ to compact grid patterns in $G_n^*$. We will refer to the algorithm as the $\gamma$ algorithm, and refer to the function computed by it as $\gamma:W_n\to G_n^*$. We first define the following function:
\begin{gather*}
    \beta:\{\sigma_a, \sigma_a^{-1}: 1\leq a\leq n - 1\}\to\{i: 1\leq i\leq n - 1\}\\
    \sigma_a^{\pm 1}\mapsto a.
\end{gather*}
Intuitively, this function finds the number of the first strand involved in the crossing for any crossing $\sigma_a^{\pm1}$.

We use algorithms which we refer to as the type A and type B algorithms in the $\gamma$ algorithm. These algorithms are listed after the $\gamma$ algorithm.

\hfill\break
$\gamma$ \textbf{Algorithm:}

\hfill\break
Let $w\in W_n, w = x_1^{(1)}x_2^{(1)}\cdots x_m^{(1)}$ where each $x_i^{(1)}$ represents the generator at position $i$ in $w$. The algorithm will return $g_r$ (for some $r\in\mathbb{N}$), a compact grid representation such that $\psi(g_r) = w$.
\begin{enumerate}
    \item Define $g_0 = []$.
    \item Determine if $w=1$. If so, return $g_0$.
    \item Define $w_1 = w$ and compute $\beta(x_1^{(1)})$.
    \begin{enumerate}
        \item If $\beta(x_1^{(1)})$ is odd, call the type A algorithm on $(w_1,[])$. Call its return $(w_2,\delta_1)$. 
        \hfill\break
        /// $\delta_1$ will correspond to a prefix of $w_1$ and $w_{2}$ will correspond to the rest of $w_1$.
        \item If $\beta(x_1^{(1)})$ is even, call the type B algorithm on $(w_1,[])$. Call its return $(w_2,\delta_1)$.
    \end{enumerate}
    \item Define $g_1 = [\delta_1]$.
    \item For $r>1$:
    \begin{enumerate}
        \item Determine if $w_r = 1$. If so, return $g_{r-1}$.
        \item If a type A (type B) step was executed at $r-1$, execute an type A (type B) step on input $(w_r, \delta_{r-1})$. Call the return $(w_{r+1},\delta_r)$.
        \item Define $g_r = g_{r-1}\bullet[\delta_r]$.
    \end{enumerate}
\end{enumerate}

\hfill\break
\textbf{Type A Algorithm:} 

\hfill\break
Given a pair $(w_r, \delta_{r-1})$ such that $w_r\in W_n$ where $w_r = x_1^{(r)}x_2^{(r)}...x_m^{(r)}$ and $\delta_{r-1}$ is a list of generations, we will return a pair $(w_{r+1}, \delta_r)$ such that $w_{r+1}$ is a suffix of $w_r$ and $\delta_r$ is a list of generations that corresponds to a prefix of $w_r$:
\begin{enumerate}
    \item  If $w_r = 1:$ Define $w_{r+1} =1$ and $\delta_r = \delta_{r-1}$, go to step 6.
    \item Determine the maximal $i\leq m$ such that both conditions are met:
    \begin{itemize}
    \item For all $1\leq k\leq i$, $\beta(x_k)$ is odd.
    \item For all $1\leq j < k \leq i$, $\beta(x_j) < \beta(x_k)$.
\end{itemize}

    \hfill\break
    ///In the following step, we build the list of cells to return by determining where in the list the crossings in $x_1^{(r)}x_2^{(r)}\cdots x_i^{(r)}$ should be put:
    \item For $1\leq t\leq \lceil\frac{n}{2}\rceil$ ($t$ will correspond to a cell number):
    \begin{enumerate}
        \item If there exists an $h$ with $1\leq h\leq i$ such that $\lceil\frac{\beta(x_h)}{2}\rceil = t$, define $C_{r, t} = [r_{a}^{(a)}, l_{a+1}^{(a+1)}]$ if $x_h^{(r)}$ is $\sigma_a$ and $C_{r, t} = \overline{[r_{a}^{(a)}, l_{a+1}^{(a+1)}]}$ if $x_h^{(r)}$ is $\sigma_a^{-1}$. 
    \item If (a) is not satisfied and $\lceil\frac{n}{2}\rceil > t$, define $C_{r, t} = [s_{2t - 1}^{(2t-1)}, s_{2t}^{(2t)}]$
    \item If (a) and (b) are not satisfied:
    \begin{enumerate}
        \item If $n$ is odd: Define $C_{r, t} = [s_n^{(n)}, n_{n+1}^{(\varnothing)}]$.
        \item If $n$ is even: Define $C_{r, t} = [s_{2t - 1}^{(2t-1)}, s_{2t}^{(2t)}]$.
    \end{enumerate}
    \end{enumerate}
    \item Define $w_{r+1} = x_{i+1}^{(r)}\cdots x_m^{(r)}$.
    \item Define $\delta_r = C_{r, 1}C_{r, 2} \cdots C_{r, \lceil\frac{n}{2}\rceil}$.
    \item Return $(w_{r+1}, \delta_r)$.
\end{enumerate}

\hfill\break
\textbf{Type B Algorithm:}

\hfill\break
Given a pair $(w_r, \delta_{r-1})$ such that $w_r\in W_n$ where $w_r = x_1^{(r)}x_2^{(r)}...x_m^{(r)}$ and $\delta_{r-1}$ is a generation, we will return a pair $(w_{r+1}, \delta_r)$ such that $w_{r+1}$ is a suffix of $w_r$ and $\delta_r$ is a list of generations that corresponds to a prefix of $w_r$:
\begin{enumerate}
    \item If $w_r = 1:$ Define $w_{r+1} = 1$ and $\delta_r = \delta_{r-1}$, go to step 6.
    \item Determine the maximal position $i\leq m$ in $w_r$, such that both conditions are met:
    \begin{itemize}
        \item For all $1\leq k\leq i$, $\beta(x_k)$ is even.
        \item For all $1\leq j < k \leq i$, $\beta(x_j) < \beta(x_k)$.
    \end{itemize}
    If such an $i$ does not exist, define $i=0$.

    \hfill\break
    ///In the following step, we build the list of cells to return by determining where in the list the crossings in $x_1^{(r)}x_2^{(r)}\cdots x_i^{(r)}$ should be put:
    \item For $1\leq t\leq \lfloor\frac{n}{2}\rfloor + 1$ ($t$ will correspond to a cell number):
    \begin{enumerate}
    \item Define $C_{r, 0} = [n_0^{(\varnothing)}, s_1^{(1)}]$ if $t = 1$.

    \hfill\break
    ///Step (a) differs from the type A algorithm here because the cells output by the type B algorithm contain crossings on strands $a$ and $a+1$ such that $a$ is even.
    \item If there exists a $h$ such that $1\leq h\leq i$ and $\lfloor\frac{\beta(x_h)}{2}\rfloor = t$, define $C_{r, t} = [r_{a}^{(a)}, l_{a+1}^{(a+1)}]$ if $x_h^{(r)}$ is $\sigma_a$ and $C_{r, t} = \overline{[r_{a}^{(a)}, l_{a+1}^{(a+1)}]}$ if $x_h^{(r)}$ is $\sigma_a^{-1}$. 
    \item If no such $h$ exists, $t\neq 1$, and $\lfloor\frac{n}{2}\rfloor + 1 > t$, define $C_{r, t} = [s_{2t - 2}^{(2t-2)}, s_{2t - 1}^{(2t-1)}]$
    \item If (a), (b), and (c) are not satisfied:
    \begin{enumerate}
        \item If $n$ is even: Define $C_{r, t} = [s_n^{(n)}, n_{n+1}^{(\varnothing)}]$.
        \item If $n$ is odd: Define $C_{r, t} = [s_{2t - 2}^{(2t-2)}, s_{2t - 1}^{(2t-1)}]$
    \end{enumerate}
    \end{enumerate}
    \item Define $\delta_r = C_{r, 1}C_{r, 2}\cdots C_{r, \lfloor\frac{n}{2}\rfloor + 1}$.
    \item Define $w_{r+1} = x_{i+1}^{(r)}\cdots x_m^{(r)}$.
    \item Return $(w_{r+1}, \delta_r)$.
\end{enumerate}
\begin{lemma}\thlabel{halts}
    The $\gamma$ algorithm halts on any input.
\end{lemma}
\begin{proof}
     It suffices to show that given any appropriate pair $(w, \delta)$ with $w\neq1$, the return of the type A (type B) algorithm run on the output of the type B (type A) algorithm is a pair $(v,\delta')$ such that $\length(v)<\length(w)$. 

     Let $(w, \delta)$ be such that $w\neq1$. Then either $\beta(x_1^{(1)})$ is even or it is odd. First consider the case where $\beta(x_1^{(1)})$ is even. The type B (type A) algorithm has an $i\in[m]$ such that (2)(a) is met. By possibly choosing a smaller $i$, we may ensure condition (2)(b) is met for whichever algorithm the $i\in[m]$ was chosen for, as $i=1$ satisfies both conditions but may not be maximal. Notice that step (3) must always terminate by construction. Thus, step (5) is reached, and $\length(v) = \length(w)-i<\length(w)$ (in the case where $i\in[m]$ was chosen for the type A (type B) algorithm, step (5) is reached after the type B (type A) algorithm is run on $(w,\delta)$ and returns $(w,\delta)$ and then after the type A (type B) algorithm is run on $(w,
     \delta)$).
\end{proof}
The following is clear from the construction of the $\gamma$ algorithm:
\begin{remark}
    For any $w\in W_n, \gamma(w)$ is a compact grid representation of $w$.
\end{remark}
\subsection{Structure of Compact SCA Representations}
In this section, we will prove results about the structure of the grid patterns the  $\gamma$ algorithm returns, and the structure of SCA representations of a word $w$ in relation to $\gamma(w)$. Recall the ordering introduced in \ref{height}.
\begin{definition}\thlabel{orderingdef}
    Let $g$ be a grid pattern. If there exists a cell $C_{i,j}$ such that there exists no grid pattern $h$ with $\varphi(g) = \varphi(h)$ and a cell $C_{a, b}$ of $h$ such that at least one of $a<i, b<j$ is true and $\varphi(g^{<C_{i,j}}) = \varphi(h^{C_{a,b}})$, we say that $C_{i,j}$ cannot be \textbf{placed in a lower position}. 
\end{definition}
\begin{definition}
    We say a compact SCA pattern $g$ \textbf{has trim} if there exists some $t\in\mathbb{N}^+,t\leq\length(g)$ such that for all $i\geq t$, $\psi(\delta_t) = 1$ or if $\psi(\delta_1) = 1$.
\end{definition}
\begin{definition}
    We call a compact SCA pattern $g$ \textbf{trimmed} if $\delta_0$ and $\delta_{\length(g)}$ both contain crossings.
\end{definition}
Notice that for any compact SCA pattern $g$ such that $\psi(g)\neq1$, there exists a trimmed sublist $h$ of $g$. By \cite[Lemma 8.5]{gliders}, $h$ is an SCA pattern. Additionally, $\psi(h) = \psi(g)$.
\begin{lemma}\thlabel{irbm}
    Let $w\in W_n$ for some $n\geq 3$, $w\neq1$. The word $w$ is compactly SCA-representable if and only if $\gamma(w)$ is an SCA pattern.
\end{lemma}
\begin{proof}
    By the construction of the $\gamma$ algorithm, $\gamma(w)$ is a compact grid representation of $w$. Suppose $w$ is compactly SCA-representable. For contradiction, $\gamma(w)$ is not an SCA pattern. Then, for any compact SCA representation $g$ of $w$, $g\neq\gamma(w)$. Suppose $g$ is a trimmed compact SCA representation of $w$. Let $g = [\delta_1,\cdots,\delta_m]$ and $\gamma(w) = [\delta_1',\cdots,\delta_t']$. Then, by the assumption that $g$ is trimmed and the construction of $\gamma(w)$, $\delta_1$ and $\delta_1'$ are of the same type. By compactness, this implies $\delta_i$ and $\delta_i'$ are of the same type for every $1\leq i\leq \min\{m,t\}$. Thus, a cell $C_{i,j}$ in $g$ contains $i$ strands if and only if the cell $C_{i,j}'$ in $\gamma(w)$ contains $i$ strands, for all $i\in\{0,1,2\}$ by compactness. Let $C_{i,j}$ be the first cell of $g$ (under the ordering defined in \ref{height}) at which $g$ and $\gamma(w)$ differ. Then, one of $C_{i,j}$ and $C_{i.j}'$ must contain a crossing, and the other must contain two straight strands. By the construction of $\gamma(w)$, $C_{i,j}'$ contains the crossing. Because $\psi(g) = \psi(\gamma(w)) = w$, there must exist a $C_{a,b}>C_{i,j}$ that contains the crossing which $C_{i,j}'$ contains. However, by the fact that if $\delta_i$ is type A (type B) $\delta_{i+1}$ must be type A (type B), $a> i+1$. Thus, $\psi(\delta_{i+1}) = 1$, and for all $C_{k,l}$ with $C_{i,j}\leq C_{k,l}<C_{a,b}$, $\psi([C_{k,l}]) = 1$. Because $\delta_{i+1}$ is a compact generation containing only straight strands, the parents of $C_{a,b}$ must be of the form $([\cdots,s_u^{(u)}],[s_{u+1}^{(u+1)},\cdots])$, so all turning rules of $g$ also contain the bit $\frac{T}{SS}$. 
    We now consider two cases.

    \hfill\break
    \textbf{Case 1:} $C_{i,j}$ does not contain strand $n$.

    By the compactness of $g$, $\gen(C_{i,j},C_{i,j+1})$ is a cell containing two straight strands. Thus, all turning rules of $g$ have bit $\frac{S}{SS}$, a contradiction.

    \hfill\break
    \textbf{Case 2:} $C_{i,j}$ contains strand $n$.

    We consider two subcases:

    \textbf{Subcase A:} $n > 3$

    Notice that generations $\delta_{i+1},\delta_{i+2}$ must contain the turning configuration (in the sense of \cite[Section 3]{gliders}):
\begin{center}
    \includegraphics[scale=0.5]{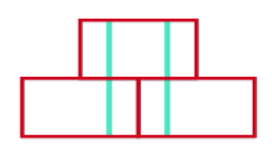}
\end{center}
    Thus, all turning rules of $g$ contain the bit $\frac{S}{SS}$, a contradiction.

    \textbf{Subcase B:} $n = 3$

    Notice that generations $\delta_i,\delta_{i+1}$ must be of the form:
\begin{center}
        \includegraphics[scale=0.5]{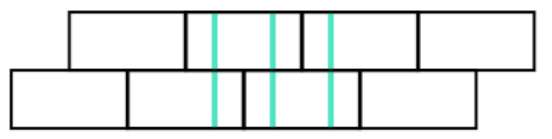}
\end{center}
     Thus, all turning rules of $g$ contain the bit $\frac{S}{SS}$, a contradiction.

Now suppose $\gamma(w)$ is an SCA pattern. Then, $\psi(\gamma(w)) = w$, so $w$ is compactly SCA-representable.
\end{proof}

\begin{proposition}\thlabel{idealRep}
    Let $w\in W_n$. $w$ is compactly SCA-representable if and only if $\gamma(w)$ is an SCA pattern.
\end{proposition}
\begin{proof}
    By the construction of the $\gamma$ algorithm, the backwards direction is trivial.

    Suppose $w$ is compactly SCA-representable. If $w=1$, $\gamma(w) = []$, and $\gamma(w)$ is an SCA pattern. We now consider two cases:
    
\hfill\break
\textbf{Case 1:} $w\neq1$ and $n = 2$
    
    For contradiction, assume the claim is false. Then, for any SCA representation $g$ of $w,$ $g\neq\gamma(w)$. Let $g = [\delta_1,\cdots,\delta_m]$ be a trimmed compact SCA representation of $w$, and let $\gamma(w) = [\delta_1',\cdots,\delta_k
']$. Then, there exists a least $i\leq\min\{k,m\}$ such that $\delta_i\neq\delta_i'$. Notice that $\delta_i,\delta_i'$ each contain at most one crossing. By compactness, $\delta_i = C_{i,0},\delta_i'=C_{i,0}'$ where $C_{i,0},C_{i,0}'$ both contain two strands. By the definition of a representation and the minimality of $i$, $\delta_i$ and $\delta_i'$ cannot both contain crossings, so by the construction of $\gamma$, $C_{i,0} = [s_1^{(1)},s_2^{(2)}]$. By the choice of $g$ as a trimmed pattern, $i>1$. Thus, because $g$ is compact, $\delta_{i-1} = [[n_0^{(\varnothing)},s_1^{(1)}][s_2^{(2)},n_3^{(\varnothing)}]]$.
We may visualize the sublist $[\delta_{i-1},\delta_i]$ of $g$ as follows:
\begin{center}
    \includegraphics{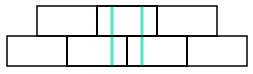}
\end{center}

This implies that $g$ has the bit $\frac{S}{SS}$ in all turning rules, so no generation after $\delta_{i-2}$ can contain a crossing. This contradicts the choice of $g$ as a trimmed pattern. Thus, $\gamma(w)$ is an SCA pattern.

\hfill\break
\textbf{Case 2:} $w\neq1$ and $n\geq 2$

This case follows from \thref{irbm}.
\end{proof}
\begin{lemma}\thlabel{formOfReps}
    Let $w\in W_n, n\geq 3$ be a compactly SCA-representable word with $w = ac^kd$ for some $k\geq 3, c\neq 1$. Then there exists a compact SCA representation $g$ of $w$ such that one of the following holds:
    \begin{enumerate}
        \item  $g = xy^kz$ and $\psi(x) = a, \psi(y) = c, \psi(z) = d$.
        \item $g = tu^{k-1}v,$ and $ u\neq[], \psi(t)\subseteq ac,\psi(u)\subset cc,$ $\psi(v)\subseteq cd$.
    \end{enumerate}
\end{lemma}
\begin{proof}
    By \thref{idealRep}, $g = \gamma(w)$ is a compact SCA representation of $w$. There are a few cases to consider:
    \begin{description}
        \item Case 1: There exists $x'\subseteq g$ such that $\psi(x') = a$, $x' = [\delta_1,\delta_2,\cdots,\delta_m]$ and there exists $y'\subseteq g, y' = [\delta_{m+1},\cdots,\delta_r]$ such that $\psi(y') = c$.
        \item Case 2: There exists a $t'\subseteq g$ such that $\psi(t') = a$, $t' = [\delta_1,\delta_2,\cdots,\delta_m]$ and there exists $u'\subseteq g, u' = [\delta_{m+1},\cdots,\delta_r]$ such that $\psi(u') \subset cc$ and $\psi([\delta_{m+1},\cdots,\delta_{r - 1}])\subset c$.
        \item Case 3: There exists a $t'\subseteq g, t = [\delta_1,\cdots,\delta_m]$ such that $\psi(t) = ac$ but the conditions for case 1 are not met.
        \item Case 4: There exists $t\subseteq g$ such that $t = [\delta_1,\cdots,\delta_m], \varphi(t) \subseteq ac, \psi([\delta_1,\cdots,\delta_{m-1}])\subset a$ and $u\subseteq g, u = [\delta_{m+1},\cdots,\delta_r]$ such that $\psi([\delta_{m+1},\cdots,\delta_{r-1}])\subset c$ and $c\subset\psi([\delta_{m+1},\cdots,\delta_r])\subseteq cc$.
    \end{description}
    We will prove that in case 1, (1) holds, and in cases 2-4, (2) holds.

    \hfill\break
    \textbf{Case 1:}
    
    Choose $x$ such that $x = [\delta_1,\cdots,\delta_m]$ and $a\subseteq\psi([\delta_1,\cdots,\delta_{m+1}])\subseteq ac, \psi(x) = a$, and $m$ is maximal. Such an $x$ is guaranteed to exist by the existence of $x'$. Choose
    $y = [\delta_{m+1},\cdots, \delta_r]$ such that $\psi(y) = c$, $c\subset\psi([\delta_{m+1},\cdots,\delta_{r+1}])$, and $r$ is maximal. Such a $y$ must exist because $\varphi(y') = c$ and $k\geq 3$. 
    
    We aim to show $\delta_{m+1} = \delta_{r+1}$. Note that $\delta_{r+1}$ must contain a crossing, or else we would reach a contradiction to the maximality of $r$. Because $\psi(\delta_{m+1}), \psi(\delta_{r+1})$ are both prefixes of $c$, $\delta_{m+1}$ and $\delta_{r+1}$ must both be type A (type B) generations. Thus, $\delta_{m+1} = C_{m+1, 0}\cdots C_{m+1, k}, \delta_{r+1} = C_{r+1, 0}\cdots C_{r+1, k}$ for some $k\in\mathbb{N}^+$. Let $i$ be the least $i$ such that $\psi(C_{m+1, i})\neq 1$.
    Notice that $\psi(\delta_{m+1})$ is a prefix of $c$ by the maximality of $m$. Then, because $\delta_{r + 1}$ contains a crossing and $\psi(\delta_{r+1})$ is a prefix of $c$, $\delta_{m+1}$ must contain the same crossing, so $C_{m+1, i} = C_{r + 1, i}$. Thus, by the definition of the $\gamma$ algorithm, $C_{m + 1, j} = C_{r + 1, j}, i\leq j\leq k$, so $\delta_{m+1} = \delta_{r + 1}$ (if $C_{m+1,j}\neq C_{r+1,j}$ for some $j<i$, we would reach a contradiction, as one such cell would have to be a crossing and the other would not contain a crossing by choice of $i$).
    
    By the fact that $g$ is an SCA pattern,  the fact that $\delta_{m+1} = \delta_{r + 1}$, and the fact that $c^k$ is a subword of $w$, $xy^k$ is a subword of $g$. Then, $g = xy^kz$ where $z = [\delta_{m + k(r-m)},\cdots,\delta_{\length(g)}]$ by choice of $y$. Because $\psi(x) = a, \psi(y) = c$ and $\psi(g) = ac^kd$, this forces $\psi(z) = d$. 
    
    \hfill\break
    \textbf{Case 2:}
    
     Choose $t''$ with $t'' = [\delta_1,\cdots,\delta_m]$ such that $\varphi(t'') = a$ and $m$ is the maximal such $m$ for which this is possible. The existence of $t''$ follows from the existence of $t'$, similarly to the analogous argument in case 1. Then, there exists a $u'' = [\delta_{m+1},\cdots,\delta_r]$ such that $c\subset\psi(u'')\subset cc$, $\psi([\delta_{m+1},\cdots,\delta_{r-1}])\subset c$, and $r$ is minimal. This follows from the fact that $\psi(g) = ac^kd$ for some $k\geq 3$ and the existence of $u'$. Let $t = [\delta_1,\cdots,\delta_{m+1}], u = [\delta_{m+2},\cdots,\delta_r]$. 
     
     We aim to show that $\delta_{m+1} = \delta_{r}$. The proof will be similar to the proof of the analogous fact given in case 1. Let $\delta_{m+1} = C_{m+1, 0}\cdots C_{m+1, k}$ and let $\delta_r = C_{r,0}\cdots C_{r,w}$. Let $i$ be the least $i$ such that $\psi(C_{m+1, i})\neq 1$ and $\psi(C_{m+1, i})\cdots\psi(C_{m+1, k})$ is a prefix of $c$. Let $j$ be the least $j$ such that $\psi(C_{r,j})\neq 1$ and $\psi(C_{r,j}\cdots C_{r,w})$ is a prefix of $c$. Then, by the construction of $\gamma$ algorithm, $i = j$, $w = k$, and $C_{m+1,l} = C_{r,l}$ for all $i\leq l\leq k$. By the $\gamma$ algorithm, this implies $C_{m+1,l} = C_{r,l}$ for $1\leq l\leq i$, so $\delta_r = \delta_{m+1}$. Then, by the fact that $g$ is an SCA pattern, $\delta_{r+1} = \delta_{m+2}$.

     By choice of $u$ and the fact that $k\geq 3, \psi(u^{k-1})\subseteq c^k$. Define 
     \begin{gather*}
         v = [\delta_{(k-1)(r-(m+2))+m+1},\cdots,\delta_{\length(g)}].
     \end{gather*}
    Then, $\psi(v)\subseteq cd$ and $g = tu^{k-1}v$.
     
     \hfill\break
     \textbf{Case 3:}

     Let $t''=[\delta_1,\cdots,\delta_m]$ with $m$ minimal such that  $\psi(t) = ac$. This is possible by choice of $t'$, similarly to cases 1 and 2. Let $\delta_q = C_{q,0}\cdots C_{q,w_q}$ for all $q\in[m]$. Let $l>m$ be the least $l$ such that $\delta_l\neq 1$, and let $r\geq l$ be the greatest $r$ such that $\psi([\delta_l,\cdots,\delta_r])\subseteq c\subset\psi([\delta_l,\cdots,\delta_{r+1}])\subseteq cc$. We claim that $\psi([\delta_l,\cdots,\delta_r])= c$. For contradiction, suppose otherwise. Then, there exists a $j$ such that $\varphi(C_{r+1,j})$ is the first symbol of $c$ and for some $i<j$, $C_{r+1,i}$ contains a crossing, $\psi([\delta_l,\cdots,\delta_r])\psi(C_{r+1,1})\cdots\psi(C_{r+1,i})\subseteq c$, and $c\subset \psi([\delta_m,\cdots,\delta_r])\psi(C_{r+1,1})\cdots\psi(C_{r+1,j})$. Notice then that if $\delta_{r+1}$ is type A (type B), $\delta_{l}$ is type A (type B) and $\delta_m$ is type A (type B). Thus, $n\geq 4$ and $\psi(\delta_{l+1}) = 1$, so all turning rules of $g$ have the rule $\frac{S}{SS}$. This is a contradiction, as $C_{r+1,i}$ cannot contain a crossing.

     Let $t = [\delta_1,\cdots, \delta_{m-1}],u = [\delta_m,\cdots, \delta_r]$. Then, let $v= [\delta_{(k-1)(r-m+1)},\cdots,\delta_{\length(g)}]$, so $g = xy^{k-1}z$ where $\varphi(x) = ac,\varphi(y) = c$, so $\varphi(z)$ must be $d$.

    \hfill\break
    \textbf{Case 4:}
    
    We aim to show $\delta_{m+1} = \delta_{r+1}$. Let $\delta_m = C_{m, 0}\cdots C_{m,k}$ for some $k$. Then, let $i$ be the least $i$ such that $\psi(C_{m, i})\neq 1$ and $\psi(C_{m, i})\cdots\psi(C_{m, k})\subseteq c$, similarly to the proof of the analogous fact in case 1. By choice of $\delta_m$, $\psi(C_{m, i})\cdots\psi(C_{m, k})$ is a prefix of $c$. Then, by the definition of compact, the $\gamma$ algorithm, and the fact that $\psi([\delta_{m+1},\cdots,\delta_{r-1}])\subseteq c$ and $c\subset\psi([\delta_{m+1},\cdots,\delta_r])\subseteq cc$, $C_{m, i} = C_{r, i}$. By the fact that $g$ is a compact SCA pattern and the definition of the $\gamma$ algorithm, $C_{m, j} = C_{r, j}$ for $i\leq j\leq k$. Thus, by the fact that $g$ is an SCA pattern, $\delta_{m+1} = \delta_{r+1}$. Recall that $\psi(u)^{k-1}\subseteq c^k$ but $\psi(u)^{k}\not\subseteq c^k$. Thus, $g = tu^{k-1}v$ where $v = [\delta_{(k-1)(r-m)+m},\cdots,\delta_{\length(g)}]$. Note that $\psi(v)\subseteq cd$. 
\end{proof}
We now prove a theorem which will allow us to determine if several different braids are compactly SCA representable by determining if a single word is compactly SCA representable:
\begin{theorem}
    Let $b\in B_n, n\geq 3$ such that $b$ is compactly SCA representable. If $b$ has an equivalence class representative $w = ac^kd, c\neq1$ for some $k\geq 3$, then one of the following holds:
    \begin{itemize}
        \item[1.] All braids with equivalence class representatives of the form $ac^jd, j\geq 2$ are compactly SCA representable
        \item[2.] There exist $t\subset ac$, $u\subset cc$, and $v\subset cd$ such that $w = tu^{k-1}v$ and all braids with equivalence class representatives of the form $tu^jv, j\geq 2$ are compactly SCA representable.
    \end{itemize}
\end{theorem}
\begin{proof}
    Apply \thref{formOfReps} to $w$ to get a compact SCA representative $g$ in one of the desired forms. We consider two cases:
    
    \hfill\break
    \textbf{Case 1:} $g = xy^kz$ for $\psi(x) = a, \psi(y) = c$ and $\psi(z) = d$

    We will show that (1) is satisified. Let $j\geq 2$ and let $h = xy^jz$. Notice that $h$ is compact. By \thref{trcrtransfer}, $h$ is an SCA pattern. Now, note that $\psi(h) = \psi(x)\psi(y)^j\psi(z) = ac^jd$. Thus, every word of the form $ac^jd, j\geq 1$ is compactly SCA-representable, so all braids $q\in B_n$ such that $q$ contains a word of the form $ac^jd, j\geq 2$ are compactly SCA representable. 

    \hfill\break
    \textbf{Case 2:}  $g = xy^{k-1}z$ for $\psi(x) = t \subseteq ac, \psi(y) = u \subseteq cc$ and $\psi(z) = v \subset cd$
    
    We will show that (2) is satisfied. Let $j\geq 2$ and let $h = xy^jz$.  Notice that $h$ is compact and $\psi(h) = tu^jv$. By \thref{trcrtransfer}, $h$ is an SCA pattern. Thus, all braids $q\in B_n$ such that $q$ contains a word of the form $tu^jv, j\geq 2$ are compactly SCA representable.
\end{proof}
\begin{corollary}
    The problem of whether a word $w\in W_n$ has a compact SCA representation is decidable.
\end{corollary}
\begin{proof}
    Compute $\gamma(w)$, and determine whether $\gamma(w)$ has a turning rule and a crossing rule. If $\gamma(w)$ has both a turning rule and a crossing rule, accept; otherwise reject.

    We now show that the above algorithm works. Notice immediately that by the definition of $\gamma$, $\gamma(w)$ is continuous. By \thref{idealRep}, $w$ has a compact SCA representation if and only if $\gamma(w)$ has a crossing and turning rule, so if the algorithm terminates in finite time, it decides the correct set. Notice that computing $\gamma(w)$ can be done in finite time by \thref{halts}, and notice that by Algorithm 2 of section \ref{fingridrep} and \cite[Turning Rule Algorithm]{gliders}, whether or not a finite grid pattern has a turning rule and a crossing rule is decidable. Thus, the claim is proved.
\end{proof}
In the last few results, we have assumed $n\geq 3$. We will justify this assumption by proving that all $b\in B_2$ are compactly SCA-representable.
\begin{proposition}
    Let $b\in B_2$. Then, $b$ is compactly SCA representable.
\end{proposition}
\begin{proof}
     Let $w\in W_2$ be such that $w\in b$. Apply the relations $\sigma_1\sigma_1^{-1} = 1, \sigma_1^{-1}\sigma_1 = 1$ and $x1 = 1x = x, x\in\{\sigma_1, \sigma_1^{-1}\}$ to $w$ until the result of the applications, $w'$, is $w' = x_1x_2\cdots x_m$ where $m\in\mathbb{N}^+$ and $x_i = x_j$ for all $1\leq i, j \leq m$. Then, we have three cases:
    \begin{itemize}
        \item Case 1: $x_i = 1$ for all $i\in[m]$
        \item Case 2: $x_i = \sigma_1$ for all $i\in [m]$
        \item Case 3: $x_i = \sigma_1^{-1}$ for all $i\in[m]$
    \end{itemize}
    In the first case, let $g = []$. 
    
    In the second case, let $g = [\delta_1,\delta_2,\cdots,\delta_{2m-1}]$ where $\delta_i = C_{i,1} = [r_1^{(1)}, l_2^{(2)}]$ for $i$ odd and $\delta_i = C_{i, 1}C_{i, 2} = [n_0^{(\varnothing)}, s_1^{(1)}][s_2^{(2)}, n_3^{(\varnothing)}]$ for $i$ even. Then, $g$ is compact by definition, the generic turning rule of $g$ is $1XXX00X0X$, the generic crossing rule of $g$ is $XXXX1XXXX$, and $\varphi(g) = [(\sigma_1)^m] = [w'] =b$. 
    
    In the third case, let $g = [\delta_1,\delta_2,\cdots,\delta_{2m-1}]$ where $\delta_i = C_{i,1} = \overline{[r_1^{(1)}, l_2^{(2)}]}$ for $i$ odd and $\delta_i = C_{i, 1}C_{i, 2} = [n_0^{(\varnothing)}, s_1^{(1)}][s_2^{(2)}, n_3^{(\varnothing)}]$ for $i$ even. Then, $g$ is compact by definition, the generic turning rule of $g$ is $1XXX00X0X$, the generic crossing rule of $g$ is $XXXX0XXXX$, and $\varphi(g) = [(\sigma_1^{-1})^m] = [w'] = b$.
\end{proof}

\section{Outlook}\label{outlook}
There are several directions for future work. One could attempt to determine the decidability of the set of braids in $B_n$ with SCA representations. Alternatively, one could find some classification of the SCA-representable braids in $B_n$ for small $n>2$ purely in terms of purely topological properties. One could also devise some notion of infinite SCA representations, where a braid can have a SCA representation that has infinite length. This question would require many of the definitions and machinery that has been built to be generalized.
\section*{Acknowledgments}
Thanks to Dr.~Holden for advising this project and the Rose Research Fellows program for funding it during the 2023-24 school year.

\end{document}